\documentclass[a4paper]{amsart}
\usepackage{amsfonts,amsmath,amssymb,amscd,amstext,amsbsy,latexsym}
\newtheorem{thm}{Theorem}[section]
\newtheorem{cor}[thm]{Corollary}

\newtheorem{prop}[thm]{Proposition}
\theoremstyle{definition}
\newtheorem{defn}[thm]{Definition}
\theoremstyle{remark}
\newtheorem{rem}[thm]{Remark}
\newtheorem{exm}[thm]{Example}
\numberwithin{equation}{section}
\newcommand{\Z}{{\bf Z}}
\begin{document}

\title{maximal subrings up to isomorphism of fields}%
\author{Alborz Azarang, Nasrin Parsa}
%\date{5 Jun 2020}
\maketitle

\centerline{Department of Mathematics, Faculty of Mathematical Sciences and Computer,}
\centerline{Shahid Chamran University of Ahvaz, Ahvaz-Iran}
\centerline{a${}_{-}$azarang@scu.ac.ir, nasrin-parsa@stu.scu.ac.ir}
%\centerline{ORCID ID: orcid.org/0000-0001-9598-2411}

% ----------------------------------------------------------------
\begin{abstract}
In this paper we study maximal subrings up to isomorphism of fields. It is shown that each field with zero characteristic has infinitely many maximal subrings up to isomorphism. If $K$ is an algebraically closed field and $x$ is an indeterminate over $K$, then we prove that integrally closed maximal subrings of $K(x)$ which contains $K$ are all isomorphic. In particular, if $K$ is an absolutely algebraic field, then $K(x)$ has only finitely many integrally closed maximal subrings up to isomorphism if and only if $K$ is algebraically closed. Also, we show that if $K$ is an absolutely algebraic field then $K$ has only finitely many maximal subrings up to isomorphism if and only if $K$ has only finitely many maximal subrings. We prove that if a commutative ring $R$ with zero characteristic has only finitely many maximal subrings up to isomorphism, then $U(R)\cap\mathbb{P}$ is finite, where $\mathbb{P}$ is the set of natural prime numbers. In particular, if $R$ is a commutative ring with zero characteristic and $Char(R/J(R))\neq 0$, then $R$ has infinitely many maximal subrings up to isomorphism. Maximal subrings up to isomorphism of $K[x]$, $K\times K$ and $K[x]/(x^2)$ for a field $K$ are investigated. If $R$ is a non-field maximal subring of a field $K$ and $\mathcal{A}$ is the set of all maximal subrings of $K$ which are isomorphic to $R$, then we prove that $\mathcal{A}=\{\sigma(R)\ |\ \sigma\in Aut(K)|\}$, in particular $|\mathcal{A}|\leq |Aut(K)|$. Moreover if $K$ is infinite and $|Aut(K)|<|K|$, then $R$ has at least $|K|$-many integrally closed maximal subrings up to isomorphism.
\end{abstract}

% ----------------------------------------------------------------
\section{Introduction}
\subsection{Motivation.} Ferrand and Olivier characterized minimal ring extension up to isomorphism of fields in \cite{frd}. In fact they proved that if $K$ is a field, then the minimal ring extensions of $K$ up to $K$-algebra isomorphism are either a field $E$ (where $E/K$ is a minimal finite field extension) or $K\times K$ or $K[x]/(x^2)$. Recently, Dobbs and Shapiro generalized the latter fact for integral domains and certain commutative rings too, see \cite{dobsid} and \cite{dobscr}. More interestingly, Dobbs in \cite{dobsiso}, proved that for each cardinal number $\mathfrak{c}\geq 2$, there exists a commutative ring $R$ which has exactly $\mathfrak{c}$ minimal ring extension up to $R$-algebra isomorphism. Also he proved that if $\mathfrak{c}$ is a cardinal number which is either $\leq \aleph_0$ or is of the form $2^{\alpha}$ (for some cardinal number $\alpha$), then there exists a field $E$ with exactly $\mathfrak{c}$ maximal subrings. In \cite{azarang6}, the first author, exactly determine fields with only finitely many maximal subrings (see also \cite[Theorem 4.13]{azarang4}) and also affine integral domain over a given field. Specially for a field $K$, maximal subrings of $K[x]/(x^2)$ are exactly determined in \cite[Theorem 4.10]{azarang6} (i.e., the converse of the last part of the result of Ferrand and Olivier mentioned in the above). More generally, in \cite{azomn}, Azarang and Oman studied commutative rings with infinitely many maximal subrings. In particular, they proved that either a ring $R$ has infinitely many maximal subrings or $R$ is a Hilbert ring. They also characterized exactly maximal subrings of $K\times K$, where $K$ is a field, see \cite[Theorem 3.4]{azomn} (i.e., the converse of the second part of the result of Ferrand and Olivier mentioned in the above). It is not hard to see that for each prime number $p$ and natural number $n\geq 2$, if $\mathbb{F}_p$ is a field with $p$ elements, then all of maximal subrings of the ring $R=\prod_{i=1}^n \mathbb{F}_p$, are all isomorphic, i.e., $R$ has a unique maximal subring up to isomorphism, see the proof of \cite[Proposition 3.5]{azkarm3}. In this article motivated by the above facts we are interested to study maximal subrings  up to isomorphism of fields and certain commutative rings.\\

All rings in this note are commutative with $1\neq 0$. All subrings, ring extensions, homomorphisms are unital. A proper subring $S$ of a ring $R$ is called a maximal subring if $S$ is maximal with respect to inclusion in the set of all proper subrings of $R$. Not every ring possesses maximal subrings (for example the algebraic closure of a finite field has no maximal subrings, see \cite[Remark 1.13]{azkrm}; also see \cite[Example 2.6]{azkrm2} and \cite[Example 3.19]{azkarm4} for more examples of rings which have no maximal subrings). A ring which possesses a maximal subring is said to be submaximal, see \cite{azarang3}, \cite{azkrm}  and \cite{azkarm4}. If $S$ is a maximal subring of a ring $R$, then the extension $S\subseteq R$ is called a minimal ring extension (see \cite{frd}) or an adjacent extension too (see \cite{adjex}). Minimal ring extensions first appears in \cite{gilmer}, for studying integral domains with a unique minimal overring. Next in \cite{frd}, these extensions are fully studied and some essential facts are obtained.\\

A brief outline of this paper is as follow. First in Section 2, we give some preliminaries about the structure of absolutely algebraic fields from \cite{bra}, \cite{azkrm} and \cite{azarang6} which are needed in this paper. Specially we give the definition of Steinitz numbers (or supper natural numbers) and their correspond to the structure of subfields of algebraic closure of finite fields. In Section 3, we investigate about maximal subrings up to isomorphism in fields and also in certain commutative rings. First we prove that if $E$ is a field with zero characteristic, then $E$ has infinitely many maximal subrings up to isomorphism. Next we characterize absolutely algebraic field with only finitely many maximal subrings up to isomorphism. In fact we prove that an absolutely algebraic field has only finitely many maximal subrings up to isomorphism if and only if it has only finitely many maximal subrings. We prove that for an absolutely algebraic field $K$, the field $K(x)$ has only finitely many integrally closed maximal subrings up to isomorphism if and only if $K$ is algebraically closed. We also show that if $K$ is an absolutely algebraic field and $X$ be a set of indeterminate variables over $K$ where $|X|\geq 2$ and $K(X)$ has only finitely many integrally closed maximal subrings up to isomorphism, then $K$ is algebraically closed. We prove that if $K$ is an absolutely algebraic field with infinitely many maximal subrings, then $K(x)$ has infinitely many maximal subrings up to isomorphism. For a field $K$, we show that all maximal subrings of $K(x)$ which are fields and contains $K$ are all isomorphic to $K(x)$. In particular, $K$ is algebraically closed if and only if maximal subrings of $K(x)$ which contains $K$ are $K$-isomorphic to either $K(x)$ or $K[x]_{(x)}$. We investigate about the number of maximal subrings of a ring $R$ up to $D$-algebra isomorphism, where $D$ is a certain (UFD) subring of $R$. In fact, we prove that if $\alpha=|U(R)\cap Irr(D)|$, then $R$ has at least $\alpha$ maximal subrings up to $D$-algebra isomorphism. In particular, if $K$ is a field which is not absolutely algebraic, then there exists a UFD subring $D$ of $K$ which $K$ is algebraic over $D$, $|Irr(D)|=|K|$ and $K$ has at least $|K|$-many maximal subrings up to $D$-algebra isomorphism. Consequently, we prove that if $R$ is a commutative ring with zero characteristic, then either $R$ has infinitely many maximal subrings up to isomorphism or $U(R)\cap \mathbb{P}$ is finite. In particular, each commutative ring which either contains $\mathbb{Q}$ as a subring or $Char(R/J(R))\neq 0$, has infinitely many maximal subrings up to isomorphism. If $K$ is an algebraically closed field, then we prove that $K[x]$ has finitely many maximal subrings up to isomorphism which contains $K$. Also we show that if $K$ is a field, then $K$ has only finitely many maximal subrings up to isomorphism if and only if $K\times K$ has only finitely many maximal subrings up to isomorphism. We prove that if $K$ is a field each of maximal subring of $K$ is not a field and $K$ has only finitely many maximal subrings up to isomorphism, then $K[x]/(x^2)$ has only finitely many maximal subrings up to isomorphism. If $K$ is a field, then we prove that two non-field maximal subrings $R$ and $S$ of $K$ are isomorphic if and only if there exists $\sigma\in Aut(K)$ such that $\sigma(R)=S$. In particular, if $K$ is an infinite field with $|Aut(K)|<|K|$, then $K$ has at least $|K|$-many non-isomorphic integrally closed maximal subrings.

\subsection{Notation.} Finally, let us recall some notation and definitions. As usual, let $Char(R)$, $U(R)$ and $J(R)$ denote the characteristic, the set of all units and the Jacobson radical ideal of a ring $R$, respectively. For any ring $R$, let $\Z=\mathbb{Z}\cdot 1_R=\{n\cdot 1_R\ |\ n\in \mathbb{Z} \}$, be the prime subring of $R$. We denote the finite field with $p^n$ elements, where $p$ is prime and $n\in\mathbb{N}$, by $\mathbb{F}_{p^n}$. Fields which are algebraic over $\mathbb{F}_p$ for some prime number $p$, are called absolutely algebraic fields. If $R$ is a ring and $a\in R\setminus N(R)$, then $R_a$ denotes the ring of quotient of $R$ respect to the multiplicatively closed set $\mathcal{C}=\{1,a,a^2,\ldots, a^n,\ldots\}$. If $D$ is an integral domain, then we denote the set of all non-associate irreducible elements of $D$ by $Irr(D)$.  Also, we denote the set of all natural prime numbers by $\mathbb{P}$. Suppose that $D\subseteq R$ is an extension of domains. By Zorn's Lemma, there exists a maximal (with respect to inclusion) subset $X$ of $R$ which is algebraically independent over $D$. By maximality, $R$ is algebraic over $D[X]$ (thus every integral domain is algebraic over a UFD; this can be seen by taking $D$ to be the prime subring of $R$). If $E$ and $F$ are the quotient fields of $D$ and $R$, respectively, then $X$ can be shown to be a transcendence basis for $F/E$ (that is, $X$ is maximal with respect to the property of being algebraically independent over $E$). The transcendence degree of $F$ over $E$ is the cardinality of a transcendence basis for $F/E$ (it can be shown that any two transcendence bases have the same cardinality). We denote the transcendence degree of $F$ over $E$ by $tr.deg(F/E)$. If $K$ is a field, the set of all field automorphisms of $K$ is denoted by $Aut(K)$.

\section{Preliminaries}
In this section we present some observation about the structure of subfields of the algebraic closure of finite fields, see \cite{bra}, \cite{azkrm} and \cite{azarang6}. First we have the following definition:

\begin{defn}
A Steinitz number is a symbole of the form $N=\prod_{i=1}^\infty p_i^{x_i}$, where $p_i$ is the $i$-th prime number and $x_i\in\{0,1,2,\ldots,\infty\}$. The set of all such symbols is denoted by $\mathbb{S}$.
\end{defn}

Similar to natural numbers, if $N$ and $M$ are two Steinitz numbers then $N=M$, $NM$, $N|M$ (and therefore $M/N$), $N\bigwedge M$ and $N\bigvee M$ are defined, see \cite[Section2.2]{bra}. The most usage of Steinitz numbers is to describe the structure of subfields of algebraic closure of finite fields. In fact if $q$ is a prime number and $\bar{\mathbb{F}}_q$ is the algebraic closure of $\mathbb{F}_q$, then for each Steinitz number $N$, the set
$$GF(q^N):=\bigcup_{n\in\mathbb{N},\ n|N}\mathbb{F}_{q^n}$$
is a subfield of $\bar{\mathbb{F}}_q$ and conversely, each subfield of $\bar{\mathbb{F}}_p$ is of the form $GF(q^N)$ for some Steinintz number $N$. In fact the $N\leftrightarrow GF(q^N)$ is a one-one correspondence between $\mathbb{S}$ and the set of all subfields of $\bar{\mathbb{F}}_q$. In particular, if $M$ and $N$ are two Steinintz numbers then $N|M$ if and only if $GF(q^N)$ is a subfield of $GF(q^M)$,  see \cite[Theorem 2.4]{bra}. Moreover, $GF(q^M)$ is a finite extension of $GF(q^N)$ if and only if $M/N\in \mathbb{N}$ and in this case the degree of $GF(q^M)$ over $GF(q^N)$ is $M/N$, see \cite[Theorem 2.10]{bra}. {\bf The Steinitz number of an absolutely algebraic field $K$ is denoted by $st(K)$.}\\

Now let $m\in\mathbb{N}$ and $N=\prod_{i=1}^\infty p_i^{x_i}$ be a Steinitz number. Let the Steinitz number $S$ be defined as $S=\prod_{i=1}^\infty p_i^{y_i}$, where $y_i=0$ if $x_i<\infty$ and otherwise $y_i=x_i=\infty$. Then there exists an irreducible polynomial of degree $m$ over $GF(q^N)$ if and only if $m\bigwedge S=1$; equivalently, there exists an irreducible polynomial of degree $m$ over $GF(q^N)$ if and only if for all primes $p$ whenever $p|m$, then $p^\infty\nmid N$, see \cite[Theorem 4.4]{bra}.\\

In \cite{azkrm}, for each Steinitz number $N$, the set $T:=\{n\in\mathbb{N}\ |\ n|N\}$ is called a $FG$-set. By the above notation let us remind the reader some needed results from \cite{azkrm} and \cite{azarang6}. First note that if $E$ is a field with zero characteristic or if $E$ is not algebraic over its prime subfield, then $E$ has a maximal subring. Now if $E$ is a field with nonzero characteristic, say $q$, which is algebraic over $\mathbb{F}_q$, then the existence of maximal subring of $E$ is determined exactely by the Steinitz number of $E$. In fact, if $N=\prod_{i=1}^\infty p_i^{x_i}$ is the Steinitz number of $E$, i.e., $E=GF(q^N)$, and $S=\prod_{i=1}^\infty p_i^{y_i}$, where $y_i=0$ if $x_i<\infty$ and otherwise $y_i=x_i=\infty$. Then $E$ has a maximal subring if and only if $S\neq N$; more precisely, if $S\neq N$, then for each prime number $p|N/S$, $E$ has a maximal subring $E_p$ with Steinitz number $N/p$, and conversely, each maximal subring of $E$ is of the form $E_p$ for some prime number $p|N/S$. Moreover, there exists a natural one-one correspondence between the prime divisors of $N/S$ and maximal subrings of $E$. In particular, $E$ has only finitely many maximal subrings if and only if $N/S\in \mathbb{N}$, \cite[Corollary 2.5]{azarang6}. By \cite[Theorems 3.2 and 3.6]{azarang6}, the following conditions are equivalent for a field $E$:
\begin{enumerate}
\item $E$ has only finitely many maximal subrings.
\item $E$ has a subfield $F$ which has no maximal subrings and $[E:F]$ is finite.
\item Every descending chain $\cdots\subset R_2\subset R_1\subset R_0=E$ where each $R_i$ is a maximal subring of $R_{i-1}$, $i\geq 1$, is finite.
\end{enumerate}
Moreover, if one of the above equivalent conditions holds, then $F$ is unique and contains all subfields of $E$ which have no maximal subrings. Furthermore, all chains in $(3)$ have the same length, $m$ say, and $R_m=F$, where $m$ is the sum of all powers of primes in the factorization of $[E:F]$ into prime numbers.\\

Finally, we remind the reader that if $R$ is a ring and $D$ be a subring of $R$ which is a UFD, then by the proof of \cite[Theorem 1.3]{azomn}, for each irreducible element $p$ of $D$, $R$ has a maximal subring $V_p$ which is integrally closed in $R$ and $(U(R)\cap Irr(D))\setminus\{p\}\subseteq U(V_p)$ (i.e., each prime of $D$ which is non-associate to $p$ and is invertible in $R$ is invertible in $V_p$), also see \cite{azarang7} for more general results.

\section{maximal subrings up to isomorphism of fields}
In this section we study maximal subrings up to isomorphism of certain fields and finally for certain commutative rings too. We remind the reader that if $E$ is a field and $R$ is a subring of $E$, then $R$ is a maximal subring of $E$ if and only either $R$ is a field and $E/R$ is a minimal finite field extension or $R$ is a one dimensional valuation domain (with exactly one nonzero prime ideal), see \cite[Theorem 3.3]{azarang}. Now the following is in order.

\begin{thm}\label{t1}
Let $E$ be a field with zero characteristic. Then $E$ has infinitely many maximal subrings up to isomorphism.
\end{thm}
\begin{proof}
Note that by \cite[Theorem 1.3]{azomn}, for each prime number $p$, $E$ has a maximal subring, say $R_p$, such that $1/p\notin R_p$. Since $E$ is a field, we infer that $R_p$ is a one dimensional valuation domain with unique maximal ideal $M_p$. Thus $p\in M_p$ and therefore $Char(R_p/M_p)=p$. This immediately shows that whenever $p$ and
$q$ are distinct prime numbers, then $R_p$ and $R_q$ are not isomorphic. Hence we infer that $E$ has infinitely many maximal subrings up to isomorphism.
\end{proof}

Hence by the previous theorem for studying maximal subrings up to isomorphism of fields, we may assume that the characteristic of a field is nonzero. In the following proposition we first characterize maximal subrings, up to isomorphisms, of absolutely algebraic fields.

\begin{prop}\label{t2}
Let $E$ be a field with nonzero characteristic which is algebraic over its prime subfield. Then the number of maximal subrings of $E$ coincide to the number of maximal subrings up to isomorphism of $E$.
\end{prop}
\begin{proof}
It suffices to show that if $K$ and $L$ are isomorphic subfields of $\bar{\mathbb{F}}_p$, then $K=L$. It is clear whenever $K$ or $L$ is finite. Now assume that $f$ be an isomorphism from $K$ onto $L$. Let $p^n|st(K)$, i.e., $\mathbb{F}_{p^n}\subseteq K$. Since  $\mathbb{F}_{p^n}\cong f(\mathbb{F}_{p^n})$ is a subfield of $L$, we immediately conclude that
$\mathbb{F}_{p^n}\subseteq L$ and therefore $p^n|st(L)$. This immediately implies that $st(K)|st(L)$ and therefore $K\subseteq L$, similarly $L\subseteq K$ and hence $L=K$.
\end{proof}

Therefore by the pervious proposition and \cite[Theorems 3.2 and 3.6]{azarang6}, we have the following immediate corollary.

\begin{cor}\label{t3}
Let $E$ be a field with nonzero characteristic which is algebraic over its prime subfield, then the following are equivalent
\begin{enumerate}
\item $E$ has only finitely many maximal subrings up to isomorphism.
\item $E$ has only finitely many maximal subrings.
\item $E$ has a subfield $F$ which has no maximal subrings and $[E:F]$ is finite.
\item Every descending chain $\cdots\subset R_2\subset R_1\subset R_0=E$ where
each $R_i$ is a maximal subring of $R_{i-1}$, $i\geq 1$, is
finite.
\end{enumerate}
\end{cor}

Hence it remains to study field with nonzero characteristic which is not algebraic over its prime subfield. First we have the following main and interesting result which also gives a new characterization of certain algebraically closed fields.

\begin{thm}\label{t4}
Let $K$ be a field with nonzero characteristic which is algebraic over its prime subfield. Then $K$ is algebraic closed if and only if $K(x)$ has only finitely many integrally closed maximal subrings up to isomorphism.
\end{thm}
\begin{proof}
First suppose that $K$ is an algebraically closed field, then we prove that $K(x)$ has a unique integrally closed maximal subring up to isomorphism. Assume that $R$ is an integrally closed maximal subring of $K(x)$, since $\mathbb{F}_p\subseteq R$, we conclude that $K\subseteq R$. Now note that $R$ is a one dimensional valuation domain and therefore by \cite[Theorem 66]{kap}, we deduce that either $R=K[x]_{(x-a)}$ for some $a\in K$ (note, $K$ is algebraically closed) or $R=K[x^{-1}]_{(x^{-1})}$. This immediately shows that $R$ is unique up to isomorphism (note, in any case $R\cong K[t]_{(t)}$). Conversely, assume that $K(x)$ has only finitely many integrally closed maximal subrings up to isomorphism. We show that $K$ is algebraically closed. Suppose $K$ is not algebraically closed, hence we conclude that $st(K)=\prod_{i=1}^\infty p_i^{x_i}$ and there exists $j$ such that $x_j\neq \infty$. Therefore by \cite[Theorem 4.4]{bra}, for each natural number $n$, there exists an irreducible polynomial $f_n(x)$ of degree $p_j^n$ over $K$. Hence for each natural number $n$, the ring $R_n=K[x]_{(f_n(x))}$ is an integrally closed maximal subring of $K(x)$. Now if $n\neq m$, then we claim that $R_n$ and $R_m$ are not isomorphic. For otherwise, if $R_n\cong R_m$, then we conclude that $K_n:=K[x]/(f_n(x))\cong K_m:=K[x]/(f_m(x))$ and therefore $st(K_n)=st(K_m)$. Hence $st(K)p_j^n=st(K)p_j^m$ and thus $m=n$ which is absurd. Thus $K$ is an algebraically closed field.
\end{proof}

As the following result shows, the half of the previous theorem holds for more general fields.

\begin{prop}\label{tt}
Let $K$ be a field with nonzero characteristic which is algebraic over its prime subfield and $X$ be a set of indeterminate variables over $K$ where $|X|\geq 2$. If $K(X)$ has only finitely many integrally closed maximal subrings up to isomorphism, then $K$ is algebraically closed.
\end{prop}
\begin{proof}
Assume that $K(X)$ has only finitely many integrally closed maximal subrings up to isomorphism. We show that $K$ is algebraically closed. Let $x\in X$ and put $Y=X\setminus\{x\}$. Now suppose $K$ is not algebraically closed, hence we conclude that $st(K)=\prod_{i=1}^\infty p_i^{x_i}$ and there exists $j$ such that $x_j\neq \infty$. Therefore by \cite[Theorem 4.4]{bra}, for each natural number $n$, there exists an irreducible polynomial $f_n(x)$ of degree $p_j^n$ over $K$. Clearly, $f_n(x)$ remains irreducible in $K[X]$ and therefore in $K(Y)[x]$ (note, $K[X]=K[Y][x]$ and $K[Y]$ is a UFD). Hence for each natural number $n$, the ring $R_n=K(Y)[x]_{(f_n(x))}$ is an integrally closed maximal subring of $K(X)$ with maximal ideal $M_n=(f_n(x))_{(f_n(x))}$. Now if $n\neq m$ are natural numbers, then we claim that $R_n$ and $R_m$ are not isomorphic. For otherwise, if $R_n\cong R_m$, then we conclude that $R_n/M_n\cong R_m/M_m$ and therefore $K(Y)[x]/(f_n(x))\cong K(Y)[x]/(f_m(x))$. This implies that $K_n(Y)\cong K_m(Y)$, where $K_r:=K[x]/(f_r(x))$. Therefore $K_n\cong K_m$ and thus $st(K_n)=st(K_m)$. Hence $st(K)p_j^n=st(K)p_j^m$ and thus $m=n$ which is absurd. Thus $K$ is an algebraically closed field.
\end{proof}

\begin{prop}\label{tt1}
Let $K$ be an absolutely algebraic field with infinitely many maximal subrings (up to isomorphism). If $x$ is a variable over $K$, then $K(x)$ has infinitely many maximal subrings (which are subfields of $K(x)$) up to isomorphism.
\end{prop}
\begin{proof}
Let $st(K)=\prod_{i=1}^\infty p_i^{x_i}$. Since $K$ has infinity many maximal subrings, by observation in Section 1, we conclude that the set $A:=\{i\in \mathbb{N} \ |\ 0<x_i<\infty\}$ is infinite. Now for each $i\in A$, let $E_i$ be a maximal subring of $K$ with $st(E_i)=st(K)/p_i$. Therefore $[K:E_i]=p_i$. Clearly for each $i\neq j$ in $A$, $E_i\ncong E_j$ and for each $i\in S$, $[K(x): E_i(x)]=p_i$. Thus $E_i(x)$ is a maximal subring of $K(x)$. Finally note that for $i\neq j$ in $A$ we have $E_i\ncong E_j$. This immediately shows that $E_i(x)\ncong E_j(x)$. Thus $K(x)$ has infinitely many maximal subrings up to isomorphism.
\end{proof}

Let $R$ be a ring which is a subring of the rings $S,T$ and $\phi:S\rightarrow T$ be a ring homomorphism. If $\phi|_R$ is the identity, then $\phi$ is called $R$-algebra homomorphism. In particular, if $R$ is a field and $\phi$ is an isomorphism, then we say that $S$ and $T$ are $R$-isomorphic. For more observation about maximal subrings up to isomorphism of simple transcendental extension over a field we need the following fact, see \cite[Theorem 11.3.4]{cohn}.\\

{\large\bf L\"{u}roth's Theorem}: Let $k$ be a field and $k(x)$ be a simple transcendental extension of $k$. Then every field $E$ such that $k\subset E\subseteq k(x)$ is of the form $E=k(u)$, where $u$ is transcendental over $k$.

\begin{prop}\label{tt2}
Let $K$ be a field. Then the following hold:
\begin{enumerate}
\item If $R$ is a maximal subrings of $k(x)$ which contains $k$ and is a field, then $R\cong k(x)$.
\item $k$ is algebraically closed if and only if maximal subrings of $k(x)$ which contains $k$ are $k$-isomorphic to $k(x)$ or $k[x]_{(x)}$.
\end{enumerate}
\end{prop}
\begin{proof}
For $(1)$ first note that if $R$ is a maximal subring of a ring $T$, then for each $t\in T$ either $t^2\in R$ or $R[t^2]=T$. This immediately shows that $T$ is algebraic over $R$. Thus if $R$ is a maximal subring of $k(x)$ we conclude that $k(x)$ is algebraic over $R$ and therefore $R\neq k$. Hence we are done by L\"{u}roth's Theorem.\\
For $(2)$, first assume that $k$ is algebraically closed. Let $R$ be a maximal subring of $k(x)$ which contains $k$. If $R$ is a field, then by $(1)$, $R\cong k(x)$. If $R$ is not a field then by a similar proof of Theorem \ref{t4}, we conclude that $R\cong k[x]_{(x)}$. Conversely, if $k$ is not algebraically closed field, let $p(x)$ be an irreducible polynomial over $k$ of degree $n>1$. Then clearly $R=k[x]_{(p(x))}$ is a maximal subring of $k(x)$ which contains $k$ and $R$ is not a field. Thus by our assumption we infer that $R$ is $k$-isomorphic to $k[x]_{(x)}$. Therefore if $M$ is a maximal ideal of $R$, we conclude that $R/M$ is $k$-isomorphic to $k$, i.e., $k[x]/(p(x))$ is $k$-isomorphic to $k$ which is absurd. Thus $k$ is algebraically closed field.
\end{proof}

We need the following remark for the next observations.

\begin{rem}
Let $K$ be a field and $X$ be a set of indeterminate variables over $K$. Assume that $R$ is a maximal subring of $K(X)$ which is integrally closed in $K(X)$ (i.e., $R$ is a one dimensional valuation domain) and $K\subseteq R$. Then $R$ contains a copy of $K[X]$ (which is a UFD) with quotient field $K(X)$. To see this first note that $R$ is a valuation for $K(X)$. Therefore for each $x\in X$, either $x\in R$ or $x^{-1}\in R$. Now define $\epsilon(x)=1$ whenever $x\in R$ and $\epsilon(x)=-1$ whenever $x^{-1}\in R$. If we put $Y=\{x^{\epsilon(x)}\ |\ x\in X\}$, then $K[Y]$ is isomorphic to $K[X]$ and $K[Y]\subseteq R$, and clearly $K(X)$ is the quotient field of $K[Y]$. In particular, without loss of the generality we may assume that $K[Y]\subseteq R\subseteq K(Y)$, i.e., $R$ is an overrring of $K[Y]$ which is a maximal subring of $K(Y)$.
\end{rem}

Now we prove a result for more general case instead of fields.

\begin{thm}\label{t5}
Let $R$ be a ring and $D$ be a subring of $R$ which is a UFD. Let $\alpha=|U(R)\cap Irr(D)|$. Then $R$ has at least $\alpha$ maximal subrings up to $D$-algebra isomorphism.
\end{thm}
\begin{proof}
Similar to the proof of \cite[Theorem 1.3]{azomn}, for each irreducible element $p$ of $D$ which is a unit in $R$, $R$ has a maximal subring $V_p$ which is integrally closed in $R$, $D\subseteq V_p$, $1/p\notin V_p$ and $(U(R)\cap Irr(D))\setminus\{p\}\subseteq U(V_p)$ (i.e., each prime of $D$ which is non-associate to $p$ and is invertible in $R$, is invertible in $V_p$). Therefore for each $p\neq q$ in $U(R)\cap Irr(D)$ we infer that $V_p\neq V_q$. Now we claim that $V_p$ and $V_q$ are not isomorphic as $D$-algebra too. If $\phi$ is a $D$-algebra isomorphism from $V_p$ onto $V_q$, then $q=\phi(q)$ is invertible in $V_q$ which is absurd. Thus we are done.
\end{proof}

Now we have the following result for fields.

\begin{prop}\label{t6}
Let $K$ be a field which is not absolutely algebraic field. Then there exists a UFD subring $D$ of $K$ such that
\begin{enumerate}
\item $K$ is algebraic over $D$.
\item $|Irr(D)|=|K|$.
\item $K$ has at least $|K|$-many maximal subrings up to $D$-algebra isomorphism.
\end{enumerate}
\end{prop}
\begin{proof}
Let $\Z$ be the prime subring of $K$ and $X$ be a transcendence base for $K$ over $\Z$. Now it suffices to take $D=\Z[X]$. It is clear that $K$ is algebraic over $D$ and therefore  $|D|=|K|$. This immediately implies that $|Irr(D)|=|K|$. Now we are done by Theorem \ref{t5}.
\end{proof}

Next we give some results for commutative rings.

\begin{thm}\label{t7}
Let $R$ be a commutative ring with zero characteristic. Then either $R$ has infinitely many maximal subrings up to isomorphism or $U(R)\cap \mathbb{P}$ is finite.
\end{thm}
\begin{proof}
Assume that $q_1, q_2,\ldots$ be infinitely many prime numbers which are invertible in $R$. Now similar to the proof of Theorem \ref{t5}, for each $i$, $R$ has a maximal subring $V_i$ such that all $q_j$, except $q_i$, are invertible in $V_i$. Now since each ring isomorphism is a $\mathbb{Z}$-algebra isomorphism, we immediately conclude that $V_i$ and $V_j$ are not isomorphic. Hence $R$ has infinitely many maximal subrings up to isomorphism.
\end{proof}

\begin{cor}\label{t8}
Let $R$ be a ring, which contains $\mathbb{Q}$, then $R$ has infinitely many maximal subrings up to isomorphism.
\end{cor}
\begin{proof}
It suffices to put $D=\mathbb{Z}$ in Theorem \ref{t5}.
\end{proof}

\begin{cor}\label{t9}
Let $R$ be a ring with zero characteristic and $Char(R/J(R))\neq 0$. Then $R$ has infinitely many maximal subrings up to isomorphism.
\end{cor}
\begin{proof}
Since $R$ has zero characteristic we infer that $\mathbb{Z}\subseteq R$. Now assume that $Char(R/J(R))=n$. Therefore $nR\subseteq J(R)$ and hence $n\mathbb{Z}\subseteq J(R)$. This implies that $1+n\mathbb{Z}\subseteq U(R)$ and therefore we are done by Theorem \ref{t7} (note, it is clear that $\mathbb{P}\cap (1+n\mathbb{Z})$ is infinite).
\end{proof}

In the next results we investigate about maximal subrings up to isomorphism of $K[x]$, $K\times K$ and $K[x]/(x^2)$, where $K$ is a field.

\begin{thm}
Let $K$ be an algebraically closed field. Then maximal subrings of $K[x]$ which contains $K$ are isomorphic to $K+x(x-1)K[x]$ or $K+x^2K[x]$.
\end{thm}
\begin{proof}
Assume that $K$ is algebraically closed and $R$ be a maximal subring of $K[x]$ which contains $K$. As we mentioned in the proof of Proposition \ref{tt2}, $K[X]$ is algebraic over $R$ and therefore $K\subsetneq R$. Now by \cite[Lemma 4.6]{azarang6}, we infer that $K[x]$ is integral over $R$. Therefore by \cite[Theorem 2.8]{gilmer3}, one of the following holds:
\begin{enumerate}
\item There exists a maximal ideal $M$ of $K[x]$ such that $M\subseteq R$. Since $K$ is algebraically closed we conclude that $M=(x-a)$ for some $a\in K$. Now note that clearly $K[x]=K+(x-a)$. Since $K\subset R$, we deduce that $R=K[x]$ which is impossible.
\item There exists a maximal ideal $M$ of $K[x]$ such that $M^2\subseteq R$. Since $K$ is algebraically closed we conclude that $M=(x-a)$ for some $a\in K$. Since $K\subset R$, we deduce that $K+(x-a)^2\subseteq R$. Now it is not hard to see that $K+(x-a)^2$ is a maximal subring of $K[x]$, see \cite{azarang}. Therefore $R=K+(x-a)^2$. Finally note that $R=K+(x-a)^2K[x]=K+(x-a)^2K[x-a]\cong K+t^2K[t]$ and hence we are done in this case.
\item There exist distinct maximal ideals $M_1$ and $M_2$ of $K[X]$ such that $M_1\cap M_2\subseteq R$. Since $K$ is algebraically closed we infer that $M_1=(x-a)$ and $M_2=(x-b)$ for some $a,b\in K$ and $a\neq b$. Similar to the previous part we conclude that $R=K+(x-a)(x-b)$. Finally note that
    $$R=K+(x-a)(x-b)K[x]=K+(X)(X+a-b)K[X]=K+(X)(\frac{1}{b-a}((b-a)X-1)K[X]$$
    $$=K+(\frac{Y}{b-a})(\frac{1}{b-a}(Y-1))K[Y]\cong K+t(t-1)K[t].$$

\end{enumerate}
\end{proof}

\begin{prop}\label{tt3}
Let $K$ be a field. Then $K$ has finitely many maximal subrings up to isomorphism if and only if $K\times K$ has finitely many maximal subrings up to isomorphism.
\end{prop}
\begin{proof}
Note that by \cite[Theorem 3.4]{azomn}, $R$ is a maximal subring of $K\times K$ if and only if $R=S\times K$, or $K\times S$, where $S$ is a maximal subring of $K$, or $R=\{(\sigma_1(x),\sigma_2(x))\ |\ x\in K\}$, where $\sigma_i$ are field automorphism of $K$ for $i=1,2$, also in this case in fact $R\cong K$, see the proof of \cite[Theorem 3.4]{azomn}. Hence if $K$ has only finitely many maximal subrings up to isomorphism, say $S_1,\ldots, S_n$, then we immediately conclude that $K, S_1\times K,\ldots S_n\times K$ are all of maximal subrings up to isomorphism of $K\times K$ and therefore $K\times K$ has only finitely many maximal subrings up to isomorphism. The converse is immediate and similar.
\end{proof}

\begin{exm}
Let $p$ be a prime number and $\bar{\mathbb{F}}_p$ be the algebraic closure of $\mathbb{F}_p$. As mentioned in the introduction of this paper $\bar{\mathbb{F}}_p$ has no maximal subrings and therefore maximal subrings of $T:=\bar{\mathbb{F}}_p\times \bar{\mathbb{F}}_p$ is of the form $R=\{(\sigma_1(x),\sigma_2(x))\ |\ x\in K\}$, where $\sigma_i$ are field automorphism of $\bar{\mathbb{F}}_p$ for $i=1,2$ and as we mentioned in the previous proof each of them are isomorphic to $\bar{\mathbb{F}}_p$. Finally note that since $Aut(\bar{\mathbb{F}}_p)$ is uncountable (see \cite{bra}) and for each $\sigma\neq \tau$ in $Aut(\bar{\mathbb{F}}_p)$, $R_{\sigma}:=\{(x,\sigma(x))\ |\ x\in \bar{\mathbb{F}}_p \}\neq R_{\tau}:=\{(x,\tau(x))\ |\ x\in \bar{\mathbb{F}}_p \}$ are maximal subrings of $T$, thus we conclude that $T$ has uncountably many maximal subrings all of them are isomorphic to $\bar{\mathbb{F}}_p$.
\end{exm}

\begin{thm}\label{tt4}
Let $K$ be a field where each of maximal subrings of $K$ is not a field. If $K$ has only finitely many maximal subrings up to isomorphism, then $K[x]/(x^2)$ has only finitely many maximal subrings up to isomorphism.
\end{thm}
\begin{proof}
We may assume that $T:=K[x]/(x^2)=K+K\alpha$, where $\alpha=x+(x^2)$ and therefore $\alpha^2=0$. Now by \cite[Theorem 4.10]{azarang6}, if $R$ is a maximal subring of $T$, then either $R=S+K\alpha$ or $R\cong K$ (see the proof of \cite[Theorem 4.10]{azarang6}). Hence we must prove that if $A$ and $B$ are two isomorphic maximal subrings of $K$, then $A+K\alpha$ and $B+K\alpha$ are isomorphic too. Hence assume that $f:A\rightarrow B$ be a ring isomorphism. Then clearly $f$ is a one-one homomorphism from $A$ to $K$. Since $A$ is not a field, we infer that there exists a nonzero $u\in A$ which is not invertible in $A$. Thus $A_u=A[\frac{1}{u}]=K$, and clearly $f(u)\neq 0$. Therefore $f(u)$ is invertible in $K$. Thus by \cite[Theorem 10.3.1]{cohn}, we can extended $f$ naturally to $\bar{f}$ from $K$ to $K$ (in fact $\bar{f}$ is an automorphism of $K$, since clearly $\bar{f}$ is one-one, for $K$ is a field and $f$ and therefore $\bar{f}$ are nonzero. Finally note that $B=f(A)\subseteq Im(\bar{f})$, now since $B$ is a maximal subring of $K$, we infer that either $K\cong Im(\bar{f})=B$ which is a contradiction for $B$ is not a field, or $Im(\bar{f})=K$, i.e., $\bar{f}$ is an automorphism of $K$). Now one can easily check that the map $\phi: A+K\alpha\rightarrow B+K\alpha$ which is define by $\phi(a+k\alpha)=\bar{f}(a)+\bar{f}(k)\alpha$ is a ring isomorphism and hence we are done.
\end{proof}

\begin{prop}\label{tt5}
Let $K$ be a field and $R$ be a maximal subring of $K$ which is not a field. Assume that $\mathcal{A}$ is the set of all maximal subrings of $K$ which isomorphic to $R$. Then $|Aut(K)|\geq |\mathcal{A}|$ and $\mathcal{A}=\{\sigma(R)\ |\ \sigma\in Aut(K)\}$. In particular, if $|Aut(K)|=1$, then for each two distinct non-field maximal subrings $S$ and $R$ of $K$, we have $R\ncong S$.
\end{prop}
\begin{proof}Let $\mathcal{A}=\{R_i\ |\ i\in I\}$ where for $i\neq j$ in $I$, $R_i\neq R_j$. For each $i\in I$, let $f_i:R\rightarrow R_i$ be a ring isomorphism. Similar to the proof of the previous theorem let $\bar{f}_i:K\rightarrow K$ be the extension of $f_i$. As we see in the previous proof each $\bar{f}_i$ is an automorphism of $K$ and since $R_i\neq R_j$, for $i\neq j$ in $I$, we immediately conclude that $\bar{f}_i\neq \bar{f}_j$ too. Thus $|Aut(K)|\geq |I|=|\mathcal{A}|$. Hence for each $i\in I$, $R_i=\bar{f}_i(R)$, where $\bar{f}_i\in Aut(K)$. Therefore $\mathcal{A}\subseteq \{\sigma(R)\ |\ \sigma\in Aut(K)\}$. Conversely, assume that $\sigma\in Aut(K)$, we show that $\sigma(R)\in \mathcal{A}$. It is clear that $R\cong \sigma(R)$, for $\sigma\in Aut(R)$. Hence we must prove that $\sigma(R)$ is a maximal subring of $K$. Since $R$ is not a field we conclude that $R$ is a one dimensional valuation domain. From $\sigma\in Aut(K)$, one can easily see that $\sigma(R)$ is a valuation for $K$ too and since $R\cong \sigma(R)$ we deduce that $\sigma(R)$ is a one dimensional too and therefore is a maximal subring of $K$. Thus the equality holds. The final part is evident.
\end{proof}

\begin{cor}\label{tt7}
Two distinct maximal subrings of $\mathbb{R}$ are not isomorphic.
\end{cor}
\begin{proof}
It is well known that the only ring homomorphism from $\mathbb{R}$ to $\mathbb{R}$ is the identity, therefore $|Aut(\mathbb{R})|=1$. Also note that by \cite[Remark 2.11]{azkrm}, each maximal subring of $\mathbb{R}$ is not a field. Thus we are done by the previous proposition.
\end{proof}

\begin{prop}\label{tt8}
Let $K$ be an infinite field with $|Aut(K)|<|K|$. Then $K$ has infinitely many maximal subrings up to isomorphism. In fact, $K$ has at least $|K|$-many non-isomorphic integrally closed maximal subrings.
\end{prop}
\begin{proof}
First we claim that $K$ is not an absolutely algebraic field. For otherwise, if $K$ is an absolutely algebraic field, then it is not hard to see that $Aut(K)$ is infinite (see the proof of \cite[Corollary 3.5]{azomn}) and therefore $|K|\leq |Aut(K)|$ which is absurd. Thus $K$ is not an absolutely algebraic field, therefore if $X$ is the set of all integrally closed maximal subrings of $K$, by \cite[(2) of Theorem 1.3]{azqm}, we conclude that $|X|\geq |K|$ (or see Corollary 1.5 and the proof of Theorem 1.3 of \cite{azomn}). Now assume that $\{R_i\ |\ i\in I\}$ be a complete set of all integrally closed maximal subrings up to isomorphism of $K$ (i.e., for each $i\neq j$ in $I$, $R_i\ncong R_j$ and for each $R\in X$, there exist $i\in I$ such that $R\cong R_i$). For each $i\in I$, let $X_i$ be the set of all maximal subrings of $K$ which are isomorphic to $R_i$. It is clear that $\{X_i\}_{i\in I}$ is a partition for $X$. By Corollary \ref{tt5}, for each $i\in I$, $|X_i|\leq |Aut(K)|$, therefore $|K|\leq |X|\leq |I||Aut(K)|$, which immediately shows that $|K|\leq |I|$ and hence we are done.
\end{proof}

\begin{cor}
Let $K$ be an infinite field with only finitely many integrally closed maximal subring up to isomorphism, then $|K|\leq |Aut(K)|$
\end{cor}

Hence by the notations of the proof of Theorem \ref{tt8}, if $K$ is a field which is not absolutely algebraic, then we have $|K|\leq |X|\leq |I||Aut(K)|$. Thus we have the following immediate corollaries.

\begin{cor}
Let $K$ be a field which is not absolutely algebraic and $|Aut(K)|=|K|$. If $K$ has exactly $\alpha$ integrally closed maximal subrings up to isomorphism where $\alpha<|X|$, then $|X|=|K|$, i.e., $K$ has exactly $|K|$-many integrally closed maximal subrings.
\end{cor}

\begin{cor}
Let $K$ be a field which is not absolutely algebraic, $|Aut(K)|=|K|$ and $|K|<|X|$. Then $K$ has $|X|$-many integrally closed maximal subrings up to isomorphism.
\end{cor}

Finally we conclude this paper by the following remark.

\begin{rem}
Let $F\subseteq E$ be an extension of fields. Assume that $F$ has infinitely many non-isomorphic integrally closed maximal subrings. Then the natural question raised is that " does $E$ have infinitely many non-isomorphic integrally closed maximal subrings too?". More interestingly, assume that $V_1, V_2,\ldots$ be non-isomorphic non-field maximal subrings of $F$ and $W_1, W_2,\ldots$ be extension of them to maximal subrings of $E$ (i.e., $W_i\cap F=V_i$), respectively. Is it true that $W_i$'s are non-isomorphic too? The answer to both cases is negative (even if $E/F$ is algebraic). To see this let $p$ be a prime number and $F=\mathbb{F}_p(x)$ and $E=\bar{\mathbb{F}}_p(x)$, where $\bar{\mathbb{F}}_p$ is the algebraic closure of $\mathbb{F}_p$. Then by the proof of Theorem \ref{t4}, $F$ has infinitely many non-isomorphic maximal subrings but $E$ has a unique integrally maximal subrings up to isomorphism. Which is also shows that if $V_1$ and $V_2$ are non-isomorphic non-field maximal subrings of $F$ and $W_1$ and $W_2$ are extensions of them to non-field maximal subrings of $E$, then $W_1\cong W_2$.
\end{rem}

%\begin{defn}
%Let $R$ be a ring and $I$ be a (proper) ideal of $R$, then we say $R$ has infinitely many maximal subrings up to isomorphism module $I$, if the ring $R/I$ has infinitely many maximal subrings up to isomorphism.
%\end{defn}

%\begin{prop}
%Let $R$ be a ring, then the following hold:
%\begin{enumerate}
%\item If $R$ is a zero dimensional ring with zero characteristic, then $R$ has a maximal ideal $M$, such that $R$ has infinitely many maximal subrings up to isomorphism module $M$.
%\item semi local rings??

%\item direct product (finite and infinite)??

%\end{enumerate}
%\end{prop}

\centerline{\Large{\bf Acknowledgement}}
The first author is grateful to the Research Council of Shahid Chamran University of Ahvaz (Ahvaz-Iran) for
financial support (Grant Number: SCU.MM1402.721)

% ----------------------------------------------------------------
%\bibliographystyle{amsplain}
%\bibliography{}

\end{document}